\documentclass[12pt,letterpaper]{amsart}
\usepackage{euler, epic,eepic,latexsym, amssymb, amscd, amsfonts, xypic, url, color, epsfig}

\input xy
\xyoption{all}

 \newlength{\baseunit}               
 \newcount{\numlines}                
\alph{subsection}

\newtheorem*{tmmain}{Main Theorem}
\newtheorem*{tmnl}{Theorem}
\newtheorem{tm}{Theorem}
\newtheorem{pr}[tm]{Proposition}

\newtheorem{co}[tm]{Corollary}

\theoremstyle{definition}

\newtheorem*{assume}{Assumption}
\newtheorem{hypothesis}{Hypothesis}

\theoremstyle{remark}
\newtheorem{rmk}[tm]{Remark}

\newcommand{\bbC}{\mathbf{C}}

\newcommand{\bbG}{\mathbf{G}}

\newcommand{\bbN}{\mathbf{N}}

\newcommand{\calO}{{ \mathcal O}}

\newcommand{\Spec}{ {\operatorname{Spec}}}

\begin{document}
\pagestyle{plain}
\title{Autoduality holds for a degenerating abelian variety}

\address{Dept. of Mathematics, University of South Carolina, Columbia~SC}
\email{kassj@math.sc.edu}


\author{Jesse Leo Kass}

\begin{abstract}
	We prove that  certain degenerate abelian varieties that include compactified Jacobians,   namely stable semiabelic varieties,  satisfy autoduality.  We establish this result by proving a comparison theorem that relates the associated family of Picard schemes to the N\'{e}ron model, a result of independent interest.  In our proof, a key fact is that the total space of a suitable family of stable semiabelic varieties  has rational singularities.
\end{abstract}

\maketitle

{\parskip=12pt 
In this paper we prove that certain degenerate abelian varieties satisfy autoduality, a result we now recall.  The classical statement of autoduality is a statement about an abelian variety $J_0$ with an ample divisor $\Theta_0$ that defines a principal polarization.  If  $\tau_{x_0} \colon J_0 \to J_0$ denotes  translation by a point $x_0$, then the morphism 
\begin{gather}
	J_0 \to \operatorname{Pic}^{\underline{0}}(J_0/k), \label{Eqn: Autodual} \\
	x_0  \mapsto \calO(\tau_{x_0}^{*} \Theta_0 - \Theta_0) \label{Eqn: Autodual2}
\end{gather}
from $J_0$ to the (identity component of the) Picard scheme is an isomorphism.  The Picard scheme $\operatorname{Pic}^{\underline{0}}(J_0/k)$ is the dual abelian variety, so the fact that \eqref{Eqn: Autodual} is an isomorphism implies that the autoduality theorem holds, i.e.~that $J_0$ is self-dual.

Here we construct an isomorphism analogous to \eqref{Eqn: Autodual} in which the abelian variety is replaced by a singular projective variety, namely a  stable semiabelic variety.  Stable semiabelic varieties are degenerate abelian varieties studied in the moduli theory of abelian varieties in the work of e.g.~Alexeev, Nakamura, and Olsson.  An important example of these varieties is the  compactified Jacobian of a nodal curve, or moduli space of degree $d$ rank $1$, torsion-free sheaves which are required to satisfy a  semistability condition when $X_0$ is reducible.  A stable semiabelic variety is acted upon by a natural semiabelian variety $J^{\underline{0}}$ which equals the moduli space of multidegree $0$ line bundles when $\overline{J}$ is a compactified Jacobian.

Because $\overline{J}_0$ is a (possibly reducible) projective variety, we can form the (identity component of the) Picard scheme $\operatorname{Pic}^{\underline{0}}(\overline{J}_0/k)$, and the main theorem of this paper is that there is an isomorphism from  $J_0^{\underline{0}}$ to $\operatorname{Pic}^{\underline{0}}(\overline{J}_0/k)$ that is analogous to \eqref{Eqn: Autodual}:
\begin{tmmain}[Autoduality]
	The autoduality theorem holds for stable semiabelic varieties.
\end{tmmain}
This is Corollaries~\ref{Co: AutodualityTwo}  (for stable semiablic varieties) and \ref{Co: Autodual} (for compactified Jacobians).  (See the beginning of Section~\ref{Section: Autodual} for a discussion on the relation with stable semiabelic varieties).  

The Main Theorem has special significance when the stable semiabelic variety is a compactified Jacobian $\overline{J}_{0}$ of a nodal curve $X_0$.  For such a $\overline{J}_{0}$, one consequence of the theorem is that  $\operatorname{Pic}^{\underline{0}}(\overline{J}_0/k)$ depends only on the curve $X_0$, rather than on the compactified Jacobian $\overline{J}_0$.  Recall  $\overline{J}_0$ depends on a choice of semistability condition, and different choices produce different schemes.  For example, when $X_0$ equals  two  curves  meeting in 3 nodes, one choice produces a $\overline{J}_0$ with two irreducible components, while another produces a $\overline{J}_0$ with three irreducible components.  (See \cite[Example~13.1(3)]{oda79}.)

The autoduality theorem was known when  $\overline{J}_0$ is a fine compactified Jacobian variety by work we now review.  Recall that a compactified Jacobian is said to be fine when every semistable sheaf is stable, so that $\overline{J}_{0}$ represents a natural functor.  Otherwise we say that $\overline{J}_{0}$ is coarse. The autoduality theorem was proven by Esteves--Gagn{\'e}--Kleiman \cite[Theorem (Autoduality), pages 5-6]{esteves02} when $\overline{J}_{0}$ is the compactified Jacobian of an irreducible curve.  This result was extended by  Esteves--Rocha \cite[pages 414-415]{esteves13}  to  tree-like curves and by  Melo--Rapagnetta--Viviani \cite[Theorem~C]{melo12c} to arbitrary nodal curves.  These authors also prove results for curves with worse singularities  than nodes, and their work has been generalized in various way, e.g.~to curves with planar singularities \cite{arinkin11} and to results about the compactified Picard scheme of $\overline{J}_0$ \cite{esteves05, arinkin13, melo12b}.  The result is new when $\overline{J}_0$ is a coarse compactified Jacobian or a stable semiabelic variety that is not a compactified Jacobian.

In the case of fine compactified Jacobians, the proof of autoduality we give here is different from previous proofs and runs as follows.  A given $\overline{J}_{0}$  can be realized as the closed fiber of a family $\overline{J} \to S$ over $S := \Spec( k[[t]])$ such that the generic fiber is a principally polarized abelian variety.  We first compare  the family $\operatorname{Pic}^{\underline{0}}(\overline{J}/S)/S$ with the N\'{e}ron model of its generic fiber.  (The N\'{e}ron model is an extension of the generic fiber to a $S$-scheme that satisfies a universal mapping property.)  Thus pick  a suitable resolution of singularities $\beta \colon \widetilde{J} \to \overline{J}$ and consider the pullback homomorphism $\beta^* \colon \operatorname{Pic}^{\underline{0}}(\overline{J}/S) \to \operatorname{Pic}^{\underline{0}}(\widetilde{J}/S)$.  In Proposition~\ref{Prop: RatlSing}, we prove $\overline{J}$ has rational singularities, and  this implies the differential of $\beta^*$ --- and hence $\beta^{*}$ itself --- is an isomorphism.  A  
theorem of P{\'e}pin states that $\operatorname{Pic}^{\underline{0}}(\widetilde{J}/S)$  is the identity component of the N\'{e}ron model, so we conclude that:

\begin{tmnl}[Comparison]
	$\operatorname{Pic}^{\underline{0}}(\overline{J}/S)/S$ is the identity component of the N\'{e}ron model of its generic fiber.
\end{tmnl}
This theorem is Theorem~\ref{Theorem: NeronComparison} below, and it immediately implies the Main Theorem because the universal mapping property of the N\'{e}ron model implies that the classical autoduality isomorphism of the generic fiber extends over all of $S$.  The comparison theorem is sharp in a sense described in Remark~\ref{Remark: ThmIsSharp}.

The proof just sketched deduces autoduality from the fact that $\overline{J}_0$ deforms in a family $\overline{J}/S$ such that $\overline{J}$ has rational singularities.  By contrast, in the case where $\overline{J}_0$ is a fine compactified Jacobian, the result is deduced from  a description of $\overline{J}_0$ coming from the presentation scheme in \cite{esteves02},   from  autoduality for irreducible curves in \cite{esteves13}, and from the computation of the cohomology of a universal family of sheaves, a computation done by putting $\overline{J}_0$ into a suitable miniversal family, in \cite{melo12c}.  In particular, previous work established autoduality for the compactified Jacobian using descriptions of the compactified Jacobian as a moduli space, while the present work establishes autoduality using facts about the singularities of the total space $\overline{J}$.

A word about the characteristic.  In this paper we assume:
\begin{assume}
	$R$ is a discrete valuation ring with characteristic $0$ residue field $k$.
\end{assume}
We need to make this assumption because we make use of properties of rational singularities.  There is a well-developed theory of rational singularities in characteristic zero, but not in positive characteristic (except for the case of surface singularities).  To extend the proof of the main results of this paper to allow $k$ to have positive characteristic, it would be enough to prove that the total space $\overline{J}$ admits a rational resolution and that Corollary~\ref{Corollary: CohomologyFlat} remains valid.  Based on conversations with singularity theorists, the author believes that these properties are expected to hold in positive characteristic but there is no written reference for such results.

\section*{Notation and conventions}
A  \textbf{curve} $X_0/\Spec(k)$ over a field $k$ is a $k$-scheme that is geometrically connected, geometrically reduced, 1-dimensional, and proper over $k$.  We set $g := 1-\chi(X_0, \calO_{X_0})$ equal to the \textbf{arithmetic genus}.  When $k=\overline{k}$ is algebraically closed, we say that $X_0/\Spec(k)$ is \textbf{nodal} if the completed local ring $\widehat{\calO}_{X_0, x_0}$ of $X_0$ at a point not lying in the $k$-smooth locus is isomorphic to $k[[x,y]]/(xy)$.  In general, we say that $X_0/\Spec(k)$ is nodal if $X_0 \otimes_{k} \overline{k}$ is nodal.  A \textbf{family of curves} over a scheme $T$ is a $T$-scheme $X/T$ that is proper and flat over $T$ and such that the fibers of $X \to T$ are curves.  If the fibers are nodal curves, then we say $X/T$ is a \textbf{family of nodal curves}.  We say that a family of curves $X/S$ is \textbf{regular} if $X$ is a regular scheme (i.e.~at every closed point the Zariski tangent space has dimension equal to the local Krull dimension of $X$).

\section{Comparison with the N\'{e}ron model} \label{Section: NeronComparison}
Here we prove a comparison theorem which states that a suitable family of Picard schemes $\operatorname{Pic}^{\underline{0}}(\overline{J}/S)$ is isomorphic to the identity component of the N\'{e}ron model of its generic fiber, and we use this theorem in later sections to prove autoduality.  We apply the comparison theorem when $\overline{J}/S$ is a family of stable semiabelic varieties or compactified Jacobians, but we work in slightly greater generality in this section: we work with a family $\overline{J}/S$ that satisfies  Hypothesis~\ref{Hypothesis} below, and in later sections, we prove that the hypothesis is satisfied by the families of interest.

For the remainder of this section, we fix the spectrum $S$ of a discrete valuation ring $R$ with field of fractions $K$ and residue field $k$ that we assume has characteristic zero, and let $\overline{J}/S$ be a $S$-flat and $S$-projective $S$-scheme such that the generic fiber $\overline{J}_{K}$ is a torsor for an abelian variety of dimension $g$ and the following hypothesis is  satisfied:
\begin{hypothesis} \label{Hypothesis}
	The scheme $\overline{J}$ has rational singularities, and for all $i$,  the higher direct image $R^{i}p_{*}\mathcal{O}_{\overline{J}}$ under the projection $p \colon \overline{J} \to S$ is a locally free $\mathcal{O}_{S}$-module of rank $\binom{g}{i}$ whose formation commutes with base change.
\end{hypothesis}

Because $\overline{J}/S$ satisfies Hypothesis~\ref{Hypothesis}, we can form the associated \textbf{family of Picard schemes} $\operatorname{Pic}(\overline{J}/S)/S$.  This is the $S$-scheme that represents the fppf sheafification of the functor that assigns to a $S$-scheme $T$ the set $\operatorname{Pic}(\overline{J}_{T})$ of isomorphism classes of line bundles on $\overline{J}_{T}$. The family of Picard schemes exists as a (possibly nonseparated)  $S$-group space that is locally of finite presentation over $S$.  Indeed, because the formation  of the pushforward $p_{*} \calO_{\overline{J}}$ by $p\colon \overline{J} \to S$ commutes with base change, this representability result is \cite[(1.5)]{raynaud70}.  Because  $R^{1}p_{*} \calO_{\overline{J}}$ is locally free and its formation commutes with base change, $\operatorname{Pic}(\overline{J}/S)/S$ contains the \textbf{identity component} $\operatorname{Pic}^{\underline{0}}(\overline{J}/S)/S$, an open $S$-subgroup scheme  that is of finite type and smooth over $S$ and has the property that the fibers of $\operatorname{Pic}^{\underline{0}}(\overline{J}/S) \to S$ are the identity components of the fibers of $\operatorname{Pic}(\overline{J}/S) \to S$  by \cite[Corollary~5.14, Proposition~5.20]{kleiman05}.

We compare $\operatorname{Pic}^{\underline{0}}(\overline{J}/S)$ to the \textbf{N\'{e}ron model} of its generic fiber.  The N\'{e}ron model $N/S$ of $\operatorname{Pic}^{\underline{0}}(\overline{J}_K/K)$ is a $S$-scheme that is smooth over $S$, contains $\operatorname{Pic}^{\underline{0}}(\overline{J}_K/K)$ as the generic fiber, and satisfies the N\'{e}ron mapping property; that is, for every smooth morphism $T \to S$ the natural map 
\begin{equation} \label{Eqn: NeronMapProperty}
	\operatorname{Hom}_{S}(T,N) \to \operatorname{Hom}_{K}(T_{K}, \operatorname{Pic}^{\underline{0}}(\overline{J}_K/K))
\end{equation}
is bijective.  By a theorem of N\'{e}ron  $N/S$ exists and is separated and of finite type over $S$ \cite[Corollary~2, Section~9.7]{bosch90}.  The \textbf{identity component} $N^{\underline{0}}/S$ is defined to be the complement of the connected components of the special fiber $N_{0}$ that do not contain the group identity element $e \in N_{0}(k)$.  By construction $N^{\underline{0}}$ is an open $S$-group subscheme of $N$ such that the fibers of $N^{\underline{0}} \to S$  are connected.

The identity morphism $\operatorname{id}_{K} \colon \operatorname{Pic}^{\underline{0}}(\overline{J}_K/K) \to \operatorname{Pic}^{\underline{0}}(\overline{J}_K/K)$ extends uniquely to a $S$-morphism 
\begin{equation} \label{Eqn: PicToNeron}
	\operatorname{Pic}^{\underline{0}}(\overline{J}/S) \to N^{\underline{0}}
\end{equation}
by the N\'{e}ron mapping property, and we prove:
\begin{tm}[Comparison] \label{Theorem: NeronComparison}
	The morphism \eqref{Eqn: PicToNeron} is an isomorphism.
\end{tm}
\begin{proof}
We prove this theorem by choosing a regular $S$-model $\widetilde{J}/S$  of $\overline{J}/S$, using a result of P\'{e}pin to relate the family of Picard schemes of $\widetilde{J}/S$ to the N\'{e}ron model, and then using the rational singularities hypothesis to show that $\widetilde{J}/S$ and $\overline{J}/S$ have isomorphic families of Picard schemes.

	Let $p \colon \overline{J} \to S$ be the structure morphism.  Because $\overline{J}$ has rational singularities, we can pick a resolution of singularities  $\beta \colon \widetilde{J} \to \overline{J}$ satisfying  $R^{j}\beta_{*}\calO_{\widetilde{J}}=0$ for $j>0$ and $\beta_{*}\calO_{\widetilde{J}}=\calO_{\overline{J}}$.  The Leray spectral sequence $R^{i}p_{*} \circ R^{j} \beta_{*} \calO_{\widetilde{J}} \Rightarrow R^{i+j} (p \circ \beta)_{*} \calO_{\widetilde{J}}$ thus degenerates at the $E_2$ page, so the natural homomorphisms 
	$$
		R^{i}p_{*} \calO_{\overline{J}} = R^{i}p_{*} \circ R^{0} \beta_{*} \calO_{\widetilde{J}} \to R^{i} (p \circ \beta)_{*} \calO_{\widetilde{J}}
	$$ 
	are isomorphisms.   In particular, the direct image $R^{1} (p \circ \beta)_{*}\calO_{\widetilde{J}}$ is locally free of rank $g$ and its formation commutes with base change.
	
	This shows that the hypothesis of \cite[(1.5)]{raynaud70} holds for $\widetilde{J}$, so the family of Picard schemes $\operatorname{Pic}(\widetilde{J}/S)/S$ exists as a $S$-group space that is locally of finite presentation over $S$.  The identity component $\operatorname{Pic}^{\underline{0}}(\widetilde{J}/S)$ is canonically isomorphic to the identity component of the N\'{e}ron model of its generic fiber by P\'{e}pin's result  \cite[Proposition~10.3]{pepin13}.  The claimed result is not exactly the statement of the proposition, but we can deduce it as follows. The proposition states that the identity component of the N\'{e}ron model is canonically isomorphic to the identity component of the $S$-group smoothening (in the sense of \cite[page~174]{bosch90}) of the closure of $\operatorname{Pic}^{\underline{0}}(\widetilde{J}_{K}/K)$ in $\operatorname{Pic}(\widetilde{J}/S)$.  The relevant closure is its own $S$-group smoothening because the closure is $S$-flat (as the generic fiber is dense) and the fibers of the morphism to $S$ are smooth (by \cite[Corollay~5.15]{kleiman05} and the fact that $R^{1} (p \circ \beta)_{*}\calO_{\widetilde{J}}$ satisfies the analogue of Corollary~\ref{Corollary: CohomologyFlat}).  Finally,  $\operatorname{Pic}^{\underline{0}}(\widetilde{J}/S)$ is contained in the closure because $\operatorname{Pic}^{\underline{0}}(\widetilde{J}/S)$ is smooth (and hence flat) over $S$.  In particular, $\operatorname{Pic}^{\underline{0}}(\widetilde{J}/S)$ is the identity component of the closure, deducing the desired result from P\'{e}pin's proposition.
	
	Since P\'{e}pin's result shows that the morphism 
	$$
		\operatorname{Pic}^{\underline{0}}(\widetilde{J}/S) \to N^{\underline{0}}
	$$
	extending the identity map is an isomorphism, to prove the theorem, it is enough to show that 
	\begin{gather*}
		\beta^{*} \colon \operatorname{Pic}^{\underline{0}}(\overline{J}/S) \to \operatorname{Pic}^{\underline{0}}(\widetilde{J}/S), \\
		M \mapsto \beta^{*}(M)
	\end{gather*}
	is an isomorphism.  
	
	The map $\beta^{*}$ induces   on Lie algebras is the natural homomorphism
	$$
		R^{1}p_{*} \calO_{\overline{J}} \to R^{1}(p \circ \beta)_{*} \calO_{\widetilde{J}},
	$$	
	and we already observed that this is an isomorphism.  We conclude that $\beta^*$ is \'{e}tale.  In particular, $\beta^*$ has finite fibers.  The morphism is also birational ($\beta_{K}^{*}$ is an isomorphism), so $\beta^*$ must be an open immersion by Zariski's main theorem.  Because the fibers of $\operatorname{Pic}^{\underline{0}}(\widetilde{J}/S) \to S$ are connected, the only open $S$-subgroup scheme of $\operatorname{Pic}^{\underline{0}}(\widetilde{J}/S)$ is $\operatorname{Pic}^{\underline{0}}(\widetilde{J}/S)$ itself, and so $\beta^*$ is an isomorphism.
\end{proof}

\begin{rmk} \label{Remark: ThmIsSharp}
	The comparison theorem is sharp in the following sense.  The theorem shows that the identity component of the N\'{e}ron model is isomorphic to an open $S$-subgroup scheme of $\operatorname{Pic}(\overline{J}/S)$, and one can ask if there is a larger open subgroup scheme that is isomorphic to the N\'{e}ron model.  Without additional hypotheses, no such larger subgroup scheme exists.  We demonstrate this with the following example, which is a compactified Jacobian.

	Let $S$ equal $\Spec(\bbC[t]_{(t)})$ (the localization of $\bbC[t]$ at $(t)$), $X$  the minimal regular model of $\Spec(R[x,y]/(y^2-x^3-x^2-t^2))$, and $\overline{J}/S$ the  family of degree $0$ compactified Jacobians associated to any family of ample line bundles.  (A computation shows that in this special case the semistability condition is independent of the ample line bundle.)  Observe that $\Spec(R[x,y]/(y^2-x^3-x^2-t^2))$ has a singularity at the closed point $(x, y, t)$, so $X$ is a blow up of a compactification of  $\Spec(R[x,y]/(y^2-x^3-x^2-t^2))$, and a computation shows that the special fiber $X_0$ of $X$ consists of two rational curves meeting in two nodes.  
	
	The family $X/S$ is a family of genus $1$ nodal curves with reducible special fiber, and $\overline{J}/S$ is a family of genus $1$ curves with irreducible special fiber.    Since $\overline{J}/S$ is a family of curves, $\operatorname{Pic}^{\underline{0}}(\overline{J}/S)$ is flat over $S$ and thus $\operatorname{Pic}^{\underline{0}}(\overline{J}/S)$  is equal to the closure of its generic fiber in  $\operatorname{Pic}(\overline{J}/S)$.  We can conclude that $\operatorname{Pic}^{\underline{0}}(\overline{J}/S)/S$ is the largest subgroup scheme of $\operatorname{Pic}(\overline{J}/S)$ that contains the identity component $\operatorname{Pic}^{\underline{0}}(\overline{J}/S)$ and  is isomorphic to an open subgroup scheme of the N\'{e}ron model (for any open scheme of the N\'{e}ron model has dense generic fiber by $S$-smoothness).
	
	The identity component $\operatorname{Pic}^{\underline{0}}(\overline{J}/S)$ is not, however, the N\'{e}ron model of its generic fiber because the N\'{e}ron model has disconnected special fiber.  (The elliptic curve $J_K$ has reduction type $I_2$ in Kodaira's classification \cite[Theorem~8.2]{silverman94}.)  The theorems \cite[Th\'{e}or\'{e}me~9.3]{pepin13} and \cite[Th\'{e}or\'{e}me~8.1.4]{raynaud70} suggest that one should not ask for an open subgroup of $\operatorname{Pic}(\overline{J}/S)$ isomorphic to the N\'{e}ron model, but rather for an open subgroup scheme whose maximal separated quotient is isomorphic to the N\'{e}ron model.  In the example just discussed,  $\operatorname{Pic}(\overline{J}/S)$ is separated, so again no such open subgroup scheme exists.  
\end{rmk}
	
\section{Autoduality for degenerate abelian varieties} \label{Section: DegAbVar}
Here we use the comparison theorem, Theorem~\ref{Theorem: NeronComparison}, to prove that stable semiabelic varieties, certain degenerations of principally polarized abelian varieties, satisfy an autoduality theorem.  Stable semiabelic varieties include many compactified Jacobians, but we also give a self-contained treatment of the autoduality theorem for compactified Jacobians in later sections.  The reader interested only in compactified Jacobians is advised to skip ahead to Section~\ref{Section: Autodual}.

In this section $S$ is the spectrum of a discrete valuation ring $R$ with field of fractions $K$ and residue field $k$ that we assume has characteristic zero.   

The degenerations we study are  degenerations to a \textbf{stable semiabelic variety}, a type of degeneration defined by Alekeev and Nakamura in \cite{alexeev99}. (Note: In loc.~cit.~the term ``stable quasiabelian'' is used instead, but currently ``stable semiabelic'' is more commonly used.) In proving that stable semiabelic varieties satisfy autoduality, we make use of relatively few properties of these varieties.  Alexeev and Nakamura construct stable semiabelic varieties using Mumford's construction: starting with a family of semiabelian varieties $G/S$, they form an explicit sheaf of graded $\mathcal{O}_{G}$-algebras $\mathcal{A}$, construct an action of a discrete group $Y$ on $\operatorname{Proj}(\mathcal{A})$, and then define the quotient of $\operatorname{Proj}(\mathcal{A})$ by $Y$ to be the associated family of stable semiabelic varieties $\overline{J}/S$.  In particular, the construction provides an explicit description of the local structure of $\overline{J}$, and we use that description to prove that $\overline{J}$ satisfies Hypothesis~\ref{Hypothesis} and then we prove the autoduality theorem using Theorem~\ref{Theorem: NeronComparison}.  

While we only make use of the construction in \cite{alexeev99}, to provide context, we recall some results from \cite{alexeevb}.  In that paper,  Alexeev characterizes stable semiabelic varieties as the pairs $(\overline{J}_0, J_0^{\underline{0}})$ consisting of a  (possibly reducible) projective variety $\overline{J}_0/\Spec(k)$ with an action of  a semiabelian variety $J_0^{\underline{0}}$  that satisfies the following properties:
\begin{enumerate}
	\item The dimension of each irreducible component of $\overline{J}_0$ is equal to the dimension of $J^{\underline{0}}_{0}$;
	\item There are only finitely many orbits for the $J^{\underline{0}}_{0}$-action;
	\item the stabilizer of every point of $\overline{J}_{0}$ is connected, reduced, and contained in the maximal multiplicative torus of $J^{\underline{0}}_{0}$;
	\item $\overline{J}_0$ is seminormal.
\end{enumerate}

Semiabelic varieties appear in a stable reduction theorem \cite[Theorem~5.7.1]{alexeevb} satisfied by principally polarized abelian varieties.  A principally polarized abelian variety can be identified with a triple $(\overline{J}_{K}, J^{\underline{0}}_{K}, \Theta_{K})$ consisting of an abelian variety $J^{\underline{0}}_{K}$, a $J^{\underline{0}}_{K}$-torsor $\overline{J}_{K}$, and a Cartier divisor $\Theta_{K} \subset \overline{J}_{K}$ defining a principal polarization, and after possibly passing to a ramified extension of $R$, the pair $(\overline{J}_{K}, J^{\underline{0}}_{K})$ can be extended to a pair of $S$-flat $S$-schemes  $(\overline{J}, J^{\underline{0}})$ such that $J^{\underline{0}}$ is a family of semiabelian varieties acting on $\overline{J}$ with special fiber a stable semiabelic variety.  

The extension $(\overline{J}, J^{\underline{0}})$ is not unique, but it becomes unique if one requires that $\Theta_{K}$ extends in a suitable way.  By construction, $\Theta_{K}$ extends to a family of effective Cartier divisors $\Theta \subset \overline{J}$.  The extension is unique if we require that the special fiber $\Theta_0 \subset \overline{J}_0$ is ample and does not contain an orbit of $J^{\underline{0}}_{0}$, i.e.~that $(\overline{J}_{0}, \Theta_{0}, J^{\underline{0}})$ is a family of stable semiabelic pairs.

As was explained in the introduction, the ample divisor $\Theta_{K}$ determines an autoduality isomorphism
\begin{gather} \label{Eqn: AutodualityForAbVar}
	J^{\underline{0}}_{K} \to \operatorname{Pic}^{\underline{0}}(\overline{J}_{K}) \\
	x \mapsto \calO(\tau_{x}^{*}(\Theta_{K})-\Theta_{K}), \notag
\end{gather}
where $\tau_{x}$ is translation by $x$.  We prove that this isomorphism extends for stable semiabelic varieties:

\begin{co}[Autoduality] \label{Co: AutodualityTwo}
	If $(\overline{J}, J^{\underline{0}})$ is a family of stable semiabelic varieties extending $(\overline{J}_K, J^{\underline{0}}_K)$, then \eqref{Eqn: AutodualityForAbVar} extends to an isomorphism
	$$
		J^{\underline{0}} \cong \operatorname{Pic}^{\underline{0}}(\overline{J}/S).
	$$
\end{co}
\begin{proof}
	We show that $\overline{J}$ satisfies Hypothesis~\ref{Hypothesis} and then deduce autoduality using the universal property of the N\'{e}ron model.  To see that Hypothesis~\ref{Hypothesis} is satisfied, observe that the completed local ring of  $\overline{J}$ at a closed point is isomorphic to the completed local ring of an affine toric variety, essentially by the construction of $\overline{J}$ \cite[Theorem~3.8(i)]{alexeev99} (in that theorem, it is assumed that the special fiber $J^{\underline{0}}_0$ is a multiplicative torus, but the local structure for general $J^{\underline{0}}_{0}$ is the same as for $J^{\underline{0}}_{0}$ a multiplicative torus;  see e.g.~the proof of \cite[Lemma~4.1]{alexeev99}).   Since toric singularities are rational, we conclude that $\overline{J}$ has rational singularities.  
	
	We also need to show that the direct image $R^{i}p_{*} \mathcal{O}_{\overline{J}}$ is a locally free $\calO_{S}$-module of rank $\binom{g}{i}$ whose formation commutes with base change.  By Grauert's theorem on cohomology and base change, it is enough to show $h^{i}( \overline{J}_{K}, \calO_{\overline{J}_{K}}) = h^{i}( \overline{J}_{0}, \calO_{\overline{J}_{0}}) = \binom{g}{i}$.  For $J_{K}$, this is \cite[Corollary~2, page~121]{mumford08}, and for $\overline{J}_{0}$, it is \cite[Theorem~4.3]{alexeev99}.  
	
	We immediately deduce the result by applying Theorem~\ref{Theorem: NeronComparison} (the comparison theorem) and using the N\'{e}ron mapping property.  In detail, there are  $S$-isomorphisms
	\begin{gather}
		J^{\underline{0}} \cong (N^{\vee})^{\underline{0}} \text{ by \cite[Theorem~1, page~286]{bosch90}} \label{Eqn: NeronIsoOne} \\
		\operatorname{Pic}^{\underline{0}}(\overline{J}/S) \cong N^{\underline{0}} \text{ by Theorem~\ref{Theorem: NeronComparison}} \label{Eqn: NeronIsoTwo}
	\end{gather}  
	uniquely determined by the requirement that they  restrict to the identity on the generic fiber.  Here $N^{\vee}$ is the N\'{e}ron model of $\overline{J}_K$ and $N$ is the N\'{e}ron model of $\operatorname{Pic}^{\underline{0}}(\overline{J}_{K}/K)$.
	
	The autoduality isomorphism \eqref{Eqn: AutodualityForAbVar} extends to an isomorphism 
	\begin{equation} \label{Eqn: NeronDuality}
		N \cong N^{\vee}
	\end{equation}
	by the N\'{e}ron mapping property, hence to an isomorphism on identity components, which we just identified with $\operatorname{Pic}^{\underline{0}}(\overline{J}/S)$ and $J^{\underline{0}}$ respectively.
\end{proof}

\begin{rmk}
	Observe that in Corollary~\ref{Co: AutodualityTwo} the extension $(\overline{J}, J^{\underline{0}})$ is not unique because we do not require that  $(\overline{J}_{K}, J_{K}^{\underline{0}}, \Theta_{K})$ extends to a family of   stable semiabelic \emph{pairs}).  Consequently, for a given $(\overline{J}_{K}, J^{\underline{0}}_{K})$,  there can be many extensions to a family $(\overline{J}, J^{\underline{0}})$ of stable semiabelic varieties, but by the corollary, all these extensions have isomorphic Picard varieties.
\end{rmk}

\section{Autoduality for compactified Jacobians} \label{Section: Autodual}
Here we use Theorem~\ref{Theorem: NeronComparison}, the comparison theorem, to prove that the compactified Jacobian of a nodal curve satisfies autoduality and furthermore that the autoduality isomorphism is induced by an Abel map when an Abel map exists.  There are several different constructions of compactified Jacobians in the literature, and many of these are known to produce stable semiabelic varieties (e.g.~\cite[Theorem~5.1]{alexeev04} states that the $\phi$-compactified Jacobians from \cite{oda79} are stable semiabelic varieties), and for compactified Jacobians that are stable semiabelic varieties, we could deduce autoduality using Section~\ref{Section: DegAbVar}, but here we give a self-contained proof.  

We fix the spectrum $S$ of a discrete valuation ring $R$ with field of fractions $K$ and residue field $k$ that we assume has characteristic zero.  We let $X/S$ be a family of curves, $J^{\underline{0}}/S$ the associated family of generalized Jacobians (or moduli spaces of multidegree $0$ line bundles), and  $\overline{J}/S$ an associated family of compactified Jacobians.  

Here we take a compactified Jacobian to be one of the moduli spaces constructed in \cite{simpson94}.  In other words, we fix an integer $d$ and a family $A$ of ample line bundles on $X/S$, and let $\overline{J}$ denote the $S$-projective scheme that universally corepresents the functor that assigns to a $S$-scheme $T$ the set of isomorphism classes of families of  degree $d$ rank 1, torsion-free sheaves on $X \times_{S} T$ that are semistable with respect to $A \otimes_{S} \calO_{T}$.  The scheme $\overline{J}$ exists by \cite[Theorem~1.21]{simpson94}.  (Note: the moduli space described in loc.~cit.\ includes pure sheaves that fail to have rank 1, and $\overline{J}$ is a connected component of this larger moduli space; when stability coincides with semistability, this is shown in  \cite[Section~4.2]{kass13}, and the semistable case can be treated by applying the argument in loc.~cit.~to a suitable Quot scheme.)   We say that  the special fiber $\overline{J}_0$ is \textbf{fine}  if every degree $d$ semistable rank 1, torsion-free sheaf is stable, and otherwise we say $\overline{J}_0$ is \textbf{coarse}.  

While we limit our study to the moduli spaces from \cite{simpson94}, the author expects  the argument below is valid for the other compactified Jacobians that have been constructed.  Indeed, the key results we use about $\overline{J}$ are Propositions~\ref{Prop: RatlSing}  and \ref{Proposition: DuBois}, and the proofs of these propositions remain valid  for any family of compactified Jacobians that is a moduli space of rank 1, torsion-free sheaves that either  is fine or is constructed using Geometric Invariant Theory, and to the author's knowledge,  this includes  all compactified Jacobians in the literature.

\subsection{The singularities of a compactified Jacobian} \label{Section: SingThy}
Here we prove that compactified Jacobians satisfy Hypothesis~\ref{Hypothesis} of Section~\ref{Section: NeronComparison}, i.e.~that the comparison theorem applies.  We prove the results of this section using a local description of $\overline{J}/S$ obtained from deformation theory.  When $k=\overline{k}$ is algebraically closed and $\overline{J}/S$ is a family of fine compactified Jacobians,  the completed local ring of $\overline{J}$ at a closed point $x_0 \in \overline{J}$ can be described as:
\begin{equation} \label{Eqn: DefRingInFam}
	\widehat{O}_{\overline{J}, x_0} \cong \widehat{R}[[u_1, v_1, \dots, u_n,v_n, w_1, \dots, w_m]]/(u_1 v_1 - \pi, \dots, u_n v_n - \pi)
\end{equation}
for some uniformizer $\pi \in R$ and some integers $n, m \in \bbN$.  This is \cite[Lemma~6.2]{kass09}, a result proven using the deformation theory techniques used in  \cite{kass12}.  

When $\overline{J}/S$ is a family of coarse compactified Jacobians, the Luna slice argument used in  loc.~cit.~shows that there is a multiplicative torus $\bbG_{m}^{b}$ acting on the ring appearing on the right-hand side of Equation~\eqref{Eqn: DefRingInFam} such that the torus invariant subring is isomorphic to $\widehat{\calO}_{\overline{J}, x_0}$.  Using this result, we prove:
\begin{pr} \label{Prop: RatlSing}
	$\overline{J}$ has rational singularities, and $\overline{J} \to S$ is flat.
\end{pr}
\begin{proof}
	We can assume $k=\overline{k}$ because it is enough to prove the result after passing from $R$ to its strict henselization $R^{\text{sh}}$.  With this assumption, suppose first that $\overline{J}/S$ is a family of fine compactified Jacobians.  The morphism $\overline{J} \to S$ is flat because the ring appearing in Equation~\eqref{Eqn: DefRingInFam} is the quotient of a power series ring by elements whose images in $\widehat{R}[[u_1, v_1, \dots, w_m]]/(\pi)$ form a regular sequence \cite[Corollary to Theorem~22.5]{matsumura89}.  To see that $\overline{J}$ has rational singularities, observe that $\widehat{\calO}_{\overline{J}, x_0}$ is isomorphic to the  completion of  $k[u_1, v_1, \dots, u_n, v_n, w_1, \dots, w_m]/(u_1 v_1-u_2 v_2, \dots, u_1 v_1 - u_n v_n)$, which is the coordinate ring of an affine toric variety.  (An isomorphism is determined by a choice of coefficient field $k \subset R$.) Since toric varieties have rational singularities, so does $\overline{J}$, proving the proposition when $\overline{J}$ is fine.

	When $\overline{J}/S$ is  coarse, the argument just given shows that, if $x_0 \in \overline{J}$ is a closed point, then $\widehat{\calO}_{\overline{J}, x_0}$ is the torus invariant subring of a  ring that is $R$-flat and has rational singularities.  In particular,  $\widehat{\calO}_{\overline{J}, x_0}$  has rational singularities by \cite[Corollary, page~66]{boutot} and is flat over $R$ as it is a direct summand of a flat module.
\end{proof}

From Equation~\eqref{Eqn: DefRingInFam}, we  deduce that when $k$ is algebraically closed and $\overline{J}_0$ is a fine compactified Jacobian, the completed local ring of $\overline{J}_0$ at a closed point $x_0 \in \overline{J}_0$ is of the form
\begin{equation} \label{Eqn: DefRingInIso}
	\widehat{O}_{\overline{J}_{0}, x_0} \cong k[[u_1, v_1, \dots, u_n, v_n, w_1, \dots, w_m]]/(u_1 v_1, \dots, u_n v_n).
\end{equation}
When $\overline{J}_0$ is coarse, $\widehat{O}_{\overline{J}_{0}, x_{0}}$ is isomorphic to the invariant subring of the ring appearing on the right-hand side of Equation~\eqref{Eqn: DefRingInIso} for some action of a multiplicative torus $\bbG_{m}^{b}$.  We use these descriptions to prove:

\begin{pr}	\label{Proposition: DuBois}
	$\overline{J}_{0}$ has Du Bois singularities.
\end{pr}
\begin{proof}
	We can assume $k=\overline{k}$.  When $\overline{J}_{0}$ is fine, Equation~\eqref{Eqn: DefRingInIso} shows that the completed local ring of $\overline{J}_0$ at a closed point $x_0$ is a completed product of double normal crossing singularity rings and power series rings, and such a completed product is Du Bois by \cite[Example~3.3, Theorem~3.9]{doherty08}.  When $\overline{J}_0$ is coarse, the completed local ring of $\overline{J}_0$ at a closed point is a torus invariant subring of a Du Bois local ring and hence is itself Du Bois by \cite[Corollary~2.4]{kovas99} (the left inverse hypothesis is satisfied because a torus invariant subring is a direct summand).
\end{proof}

From Proposition~\ref{Proposition: DuBois}, we deduce that $\overline{J}/S$ satisfies Hypothesis~\ref{Hypothesis}:
\begin{co} \label{Corollary: CohomologyFlat}
	The higher direct image $R^{i}p_{*} \calO_{\overline{J}}$ of $\calO_{\overline{J}}$ under the projection $p \colon \overline{J} \to S$ is a locally free $\calO_{S}$-module of rank $g \choose i$, and its formation commutes with arbitrary base chance.
\end{co}
\begin{proof}
	The fibers of $\overline{J} \to S$ have Du Bois singularities, so the result is \cite[Th\'{e}or\`{e}me~4.6]{dubois81} together with \cite[Corollary~2, page~121]{mumford08}.
\end{proof}

 \begin{rmk}
		The results in this section imply that Theorem~\ref{Eqn: PicToNeron} (the comparison theorem) applies to  $\overline{J}/S$.  When $\overline{J}$ is regular, or more generally semi-factorial, the stronger result \cite[Th\'{e}or\'{e}me~9.3]{pepin13} applies, but  $\overline{J}/S$ can fail to be semi-factorial, as the simple example of the degenerating elliptic curve from  Remark~\ref{Remark: ThmIsSharp} shows.  (There we observed that the conclusion of \cite[Th\'{e}or\'{e}me~9.3]{pepin13} does not hold.)
 \end{rmk}

\subsection{Proof of autoduality}
In the introduction we defined  the autoduality isomorphism \eqref{Eqn: Autodual} in terms of an ample divisor, but for Jacobians the isomorphism can alternatively be defined in terms of the Abel map.  If  $L_K$ is a degree $-1$ line bundle on $X_K$, then the rule 
\begin{equation} \label{Eqn: DefAbelMap}
	 x_{K} \mapsto L_{K}(x_{K})
\end{equation}
 defines a morphism $\alpha_{K} = \alpha_{L_K} \colon X_K \to \overline{J}_K$ that is the Abel map associated to $L_K$.  The pullback morphism 
\begin{equation} \label{Eqn: AutodualForCurve}
 	\alpha_K^{*} \colon \operatorname{Pic}^{\underline{0}}(\overline{J}_K/K) \to J^{\underline{0}}_K
\end{equation}
is, up to sign, the inverse of the autoduality isomorphism  \eqref{Eqn: Autodual} from the introduction.  

If $L$ is a line bundle on $X$ that extends $L_K$, then the rule $x \mapsto L(x)$ defines a morphism $\alpha$ extending $\alpha_K$ provided the fibers of $X/S$ are irreducible \cite[2.2]{esteves02},  but when $X_0$ is reducible, the expression can fail to define a morphism because $L(x)$ can fail to be semistable.    The problem of constructing a $L$ such that $L(x)$ is always semistable (i.e.~of constructing an Abel map for a reducible curve) is nontrivial.  This and related problems are studied in \cite{caporaso07b, caporaso07, coelho08, coelho10}, and for further discussion on the topic, we direct the reader to those papers.

Regardless of whether the Abel map extends,  Propositions~\ref{Prop: RatlSing} and \ref{Proposition: DuBois} show that Hypothesis~\ref{Hypothesis} holds for $\overline{J}/S$, so we can form the family $\operatorname{Pic}^{\underline{0}}(\overline{J}/S)$ of Picard schemes, and  this family is related to the generalized Jacobian by an autoduality isomorphism:

\begin{co}[Autoduality] \label{Co: Autodual}
	If $\overline{J}/S$ is a family of compactified Jacobians extending $\overline{J}_{K}$, then  \eqref{Eqn: AutodualForCurve} extends to an isomorphism
	\begin{equation} \label{Eqn: CurveAutodual}
		\operatorname{Pic}^{\underline{0}}(\overline{J}/S) \cong J^{\underline{0}}
	\end{equation}
	that is pullback by an Abel map when an Abel map is defined.
\end{co}
\begin{proof}	
	Since we have proven that $\overline{J}$ satisfies Hypothesis~\ref{Hypothesis}, we complete the proof by arguing as in the last half of the proof of Corollary~\ref{Co: AutodualityTwo}.  The autoduality isomorphism must equal  pullback by an Abel map when an Abel map is defined because both morphisms agree on the generic fiber.
\end{proof}

\subsection*{Acknowledgements} 
Jesse Leo Kass was partially sponsored by the American Mathematical Society under an AMS--Simons Travel Grant, the Simons Foundation under Grant Number 429929, and the National Security Agency under Grant Number H98230-15-1-0264.  The United States Government is authorized to reproduce and distribute reprints notwithstanding any copyright notation herein. This manuscript is submitted for publication with the understanding that the United States Government is authorized to reproduce and distribute reprints.  

This work was started while the author was a Wissenschaftlicher Mitarbeiter at the Institut f\"{u}r Algebraische Geometrie, Leibniz Universit\"{a}t Hannover.  

The author would like to Eduardo Esteves and Filippo Viviani for helpful conversations about autoduality, Karl Schwede for very helpful conversations about singularity theory, Frank Thorne and Nicola Pagani for helpful comments about exposition, and the anonymous  referees for helpful feedback.

} 

\bibliographystyle{AJPD}

\bibliographystyle{AJPD}



\providecommand{\bysame}{\leavevmode\hbox to3em{\hrulefill}\thinspace}
\providecommand{\MR}{\relax\ifhmode\unskip\space\fi MR }
\providecommand{\MRhref}[2]{%
  \href{http://www.ams.org/mathscinet-getitem?mr=#1}{#2}
}
\providecommand{\href}[2]{#2}

\end{document}